\documentclass[12pt,reqno]{amsart}
\usepackage{cite}
\usepackage{amssymb,latexsym,amsmath,amsfonts}
\usepackage{latexsym}
\usepackage{amsmath, amsthm, amscd, amsfonts, amssymb, graphicx, color}
\usepackage[mathscr]{eucal}
= 8.9in
\theoremstyle{plain}
\newtheorem{theorem}{\textbf{Theorem}}[section]
\newtheorem{lemma}[theorem]{\textbf{Lemma}}
\newtheorem{remark}[theorem]{\textbf{Remark}}

\newcommand{\R}{\mathbb{R}}
\newcommand{\N}{\mathbb{N}}

\newtheorem{definition}[theorem]{Definition}

\newtheorem{example}[theorem]{\textbf{Example}}

\numberwithin{equation}{section}
\begin{document}

\title[Relation-Theoretic... With an Application ]
{Relation-Theoretic Metrical Fixed Point Results via $w$-distance With an Application in nonlinear fractional differential equations }
\author[T. Senapati , L.K. Dey]%
{Tanusri Senapati$^{1}$, Lakshmi Kanta Dey$^{2}$}

%\thanks{}
\address{{$^{1}$\,} Tanusri Senapati,
                    Department of Mathematics,
                    National Institute of Technology
                    Durgapur,
                    West Bengal,
                    India.}
                    \email{senapati.tanusri@gmail.com}
\address{{$^{2}$\,} Lakshmi Kanta Dey,
                    Department of Mathematics,
                    National Institute of Technology Durgapur,
                    West Bengal,
                    India.}
                    \email{lakshmikdey@yahoo.co.in}

%\address{{$^{3}$\,} Ankush Chanda,
%                    Department of Mathematics,
%                    National Institute of Technology
%                    Durgapur,
%                    West Bengal,
%                    India.}
%                    \email{ankushchanda8@gmail.com}

\thanks{*Corresponding author: lakshmikdey@yahoo.co.in (L.K. Dey)}

 \subjclass[2010]{ $47$H$10$, $54$H$25$. }
 \keywords {Complete metric space, binary relation, $w$-distance, fixed point, nonlinear fractional differential equation}

\begin{abstract}
In this article, utilizing the concept of $w$-distance, we prove the celebrated Banach's fixed point theorem in metric spaces equipped with an arbitrary binary relation. Necessarily our findings unveil another direction of relation-theoretic metrical fixed point theory.   Also, our paper consists of  several non-trivial examples which signify the motivation for such investigations. Finally, our obtained results  enable us to explore the existence and uniqueness of solutions of nonlinear fractional  differential equations involving the Caputo fractional derivative.
\end{abstract}
%extend and improve certain comparable results in existing literature particularly, results of Alam and Imdad [J. Fixed Point Theory Appl. 17(4), 2015].
\maketitle

\section{\bf Introduction}
On account of the fact that the metric fixed point theory imparts a sound basis for exploring many problems in pure and applied sciences, many authors went into the possibility of altering the concepts of metric and metric spaces.  One such interesting and important motivation is to establish fixed point results in metric space endowing with an arbitrary binary relation. Exploiting the concepts of different kind binary relations such as partial order, strict order, preorder, tolerance, transitive etc. on metric space, many mathematician are doing their research during several years, see for example \cite{nie,ghod,tur3,tur4,tur5,h}. Very recently, Alam and Imdad \cite{alm} presented relation-theoretic metrical fixed point results due to famous Banach contraction principle using an amorphous relation. No doubt their results extended and improved several comparable results in existing literature but still there are some cases where we can't explain the existence of fixed point employing their results. One of the aims of this article is to present some improved and refined version of existing results using the concept of $w$-distance.  Due to reader's advantage, we need to recall some important definitions and useful results relevant to this literature. 

Throughout this article, the notations $\mathbb{Z,N,R}$, $\mathbb{R}^+$ have their usual meanings. 
 
\begin{definition}\cite{lip}
Let $X$ be a non-empty set and $\mathcal{R}$ be a binary relation defined on $X\times X$. Then, $x$ is $\mathcal{R}$-related to $y$ if and only if $(x,y)\in \mathcal{R}$.
\end{definition}
\begin{definition}\cite{mad}
A binary relation $\mathcal{R}$ defined on $X$ is said to be complete if for all $x,y\in X$, $[x,y]\in \mathcal{R}$, where  $[x,y]\in \mathcal{R}$ stands for either $(x,y)\in \mathcal{R}~\mbox{or}~ (y,x)\in \mathcal{R}$.
\end{definition}
\begin{definition}\cite{alm}
Suppose $\mathcal{R}$ is a binary relation defined on a non-empty set $X$.  Then a sequence $(x_n)$ in $X$ is said to be  $\mathcal{R}$-preserving if $$
(x_n, x_{n+1})\in \mathcal{R}~\forall n\in \mathbb{N}\cup\{0\}.$$  
\end{definition}
\begin{definition}\cite{alm}
A metric space $(X,d)$ endowed with a binary relation $\mathcal{R}$ is said to be $\mathcal{R}$-complete if every $\mathcal{R}$-preserving Cauchy sequence converges in $X$.
\end{definition}
\begin{definition}\cite{alm}
Let $X$ be a non-empty set and $f$ be a self-map defined on $X$. Then a binary relation $\mathcal{R}$ on $X$ is said to be $f$-closed if $(x,y)\in \mathcal{R}\Rightarrow (fx,fy)\in \mathcal{R}$.
\end{definition}
Here we introduce the notion of weak $f$-closed binary relation.
\begin{definition}
Let $X$ be a non-empty set and $f$ be a self-map defined on $X$. Then a binary relation $\mathcal{R}$ on $X$ is said to be weak $f$-closed if $(x,y)\in \mathcal{R}\Rightarrow [fx,fy]\in \mathcal{R}$.
\end{definition} 
It is easy to show that every $f$-closed binary relation $\mathcal{R}$ is weak $f$-closed but the converse is not true in general. To show this we present the following example.
\begin{example}
Let $X=\N$ and $\mathcal{R}$ be a binary relation defined on $X$ such that $(x,y)\in \mathcal{R}$ if $x=2m, y=2n+1$ for some $m,n\in N$. Now, we define a function $f:X\rightarrow X$ by $f(x)=x+1$ for all $x\in X$.
Then it is trivial to show that $(x,y)\in \mathcal{R}\nRightarrow (fx,fy)\in \mathcal{R}~\mbox{but}~(fy,fx)\in \mathcal{R}$. Hence, the binary relation $\mathcal{R}$ is not $f$-closed but it is weak $f$-closed.
\end{example}

\begin{definition}\cite{alm}
Let $(X,d)$ be a metric space endowed with a binary relation $\mathcal{R}$. Then, $\mathcal{R}$ is said to be $d$-self-closed if every $\mathcal{R}$-preserving sequence with  $x_n \to x$ there is a subsequence $(x_{n_k})$ of $(x_n)$ such that $(x_{n_{k}},x)\in \mathcal{R}$, for all $k\in \mathbb{N}\cup\{0\}$.
\end{definition}

%\begin{definition}\cite{sam1}
%Let $(X,d)$ be a metric space endowed with a binary relation $\mathcal{R}$. Then, a subset $S\subseteq X$ is said to be $\mathcal{R}$-directed if every  for every $x,y\in S$ there is a point $z\in S$ such that $(x,z)\in \mathcal{R}$ and  $(y,z)\in \mathcal{R}$. 
%\end{definition}
%\begin{definition}\cite{kol}
%Let $(X,d)$ be a metric space and $\mathcal{R}$ be an arbitrary relation on $X$.  For any pair of elements $x,y\in X$, a finite sequence $\{z_0,z_1,z_2,\dots,z_k\}$ in $X$ is said to be a path of length $k$ joining  $x$ to $y$ in $\mathcal{R}$ if $z_0=x, z_k=y~\mbox{and}~(z_i,z_{i+1})\in \mathcal{R}$ for each $i\in \{1,2,\dots,k-1\}$. 
%\end{definition}
For the sake of reader's perception, we recollect some notations from existing literature:
\begin{enumerate}
\item[(A)]$F(T)=\{x\in X:Tx=x\}$,
\item[(B)]$X(T,\mathcal{R})=\{x\in X:(x,Tx)\in \mathcal{R}\}.$
%\item[(C)] $Y(x,y,\mathcal{R})$: the family of all paths joining $x$ to $y$,
\end{enumerate}
% Here, we recall/bring off some results in the existing literature consistent with this article.
Before proceeding further, we record the following results. 
 \begin{theorem}(Theorem 3.1, Alam and Imdad \cite{alm}) Let $(X,d)$ be a complete metric space equipped with a binary relation $\mathcal{R}$. Suppose $T$ is a self-mapping on $X$ such that 
 \begin{enumerate}
 \item $X(T,\mathcal{R})\neq \phi$,
 \item $\mathcal{R}$ is $T$-closed,
 \item either $T$ is continuous or $\mathcal{R}$ is $d$-self-closed,
 \item there exists $k\in [0,1)$ such that
 $$d(Tx,Ty)\leq k d(x,y)~~~\forall x,y\in X ~\mbox{with}~(x,y)\in \mathcal{R}.$$
 \end{enumerate}
 Then $F(T)\neq \phi.$
 \end{theorem}
\begin{theorem}(Theorem 2.1, Ahmadullah et al.\cite{ahm}) Let $(X,d)$ be a metric space equipped with a binary relation $\mathcal{R}$. Suppose $T$ is a self-mapping on $X$ with the following conditions: 
 \begin{enumerate}
 \item there exists $Y\subseteq X, TX\subseteq Y\subseteq X$ such that $(Y,d)$ is $\mathcal{R}$-complete,
 \item $X(T,\mathcal{R})\neq \phi$,
 \item $\mathcal{R}$ is $T$-closed,
 \item either $T$ is $\mathcal{R}$-continuous or $\mathcal{R}|_Y$ is $d$-self-closed,
 \item there exists $\phi \in \Phi$ such that
 $$d(Tx,Ty)\leq \phi(M_T(x,y))~~~\forall x,y\in X ~\mbox{with}~(x,y)\in \mathcal{R},$$
 where $M_T(x,y)= \max \{d(x,y), d(x,Tx),d(y,Ty),\frac{d(x,Ty)+d(y,Tx)}{2}\}$.
 \end{enumerate}
 Then $F(T)\neq \phi.$
 \end{theorem}
Next, we would like to draw the reader's attention in another direction of metric fixed point theory.
In 1996, Kada et al. \cite{kada} introduced the idea of $w$-distance in metric spaces and established  several well-known results using this concept. They defined the $w$-distance as follows:
\begin{definition}\cite{kada}\label{w}
Let $(X,d)$ be a metric space. A function $p:X\times X\rightarrow [0,\infty)$ is said to be a $w$-distance if 
\begin{enumerate}
\item[(w1)] $p(x,z)\leq p(x,y)+p(y,z)$ for any $x,y,z\in X$,
\item[(w2)]for any $x\in X, p(x,.):X\rightarrow [0,\infty)$ is lower-semi-continuous,
\item[(w3)] for any $\epsilon>0,$ there exists $\delta>0$ such that $p(z,x)\leq \delta$ and $p(z,y)\leq \delta$ imply $d(x,y)\leq \epsilon$.
\end{enumerate}
\end{definition}
\begin{remark}
Note that a $w$-distance function $p$ may not be symmetric and also it is possible that $p(x,x)\neq 0$ for some $x$, i.e., $p(x,y)=0$ does not imply $x=y$.
\end{remark}
The readers are refereed to \cite{kada} for some examples and crucial properties of $w$-distance.

To establish fixed point results owing to $w$-distance in metric spaces equipped with arbitrary binary relation $\mathcal{R}$, we need to define the concept of $\mathcal{R}$-lower-semi-continuity (briefly, $\mathcal{R}$-LSC) of a function and then we show that notion of $\mathcal{R}$-LSC is weaker than $\mathcal{R}$-continuity as well as lower-semi-continuity.

Before defining $\mathcal{R}$-lower-semi-continuity, we look back on $\mathcal{R}$-continuity of a function defined on a metric space equipped with an arbitrary binary relation $\mathcal{R}$.

\begin{definition}\cite{alm}
Let $(X,d)$ be a metric space and $\mathcal{R}$ be a binary relation defined on $X$. A function $f:X\rightarrow X$ is said to be $\mathcal{R}$-continuous at $x$ if for every $\mathcal{R}$-preserving sequence $(x_n)$ converging to $x$, we have, 
$$f(x_n)\rightarrow f(x)~\mbox{as}~n\rightarrow \infty.$$
\end{definition}
 The notion of $\mathcal{R}$-lower-semi-continuity of a function is defined as follows:
\begin{definition}\label{dfn5}
Let $(X,d)$ be a metric space and $\mathcal{R}$ be a binary relation defined on $X$. A function $f:X\rightarrow \R\cup{ \{-\infty,\infty\}}$ is said to be $\mathcal{R}$-LSC at $x$ if for every $\mathcal{R}$-preserving sequence $(x_n)$ converging to $x$, we have, 
$$\displaystyle{\liminf_{ n\rightarrow \infty}}f(x_n)\geq f(x).$$
\end{definition}
 
The following example shows that $\mathcal{R}$-LSC is weaker than  $\mathcal{R}$-continuity.
\begin{example}
Let $X=\mathbb{R}$. Define $(x,y)\in \mathcal{R}$ if 
$x,y\in (n-\frac{1}{5},n+\frac{1}{5})~\mbox{for~ some}~n\in \mathbb{Z}$. We consider the standard metric $d$ on $X$. Let $f:X\rightarrow X$ be defined as $$f(x)=\lceil x \rceil.$$
We claim that this function is not $\mathcal{R}$-continuous but it is $\mathcal{R}$-lower semi-continuous. Let $(x_n)$ be a non-constant $\mathcal{R}$-preserving sequence converging to an integer $k$. Then there exists some $n_o\in \mathbb{N}$ such that $x_n\in (k-\frac{1}{5},k+\frac{1}{5})$ for all $n>n_0$. Now, if $x_n\rightarrow k$ from left, then $\lim_{n\rightarrow \infty}f(x_n)=k$ and if $x_n\rightarrow k$ from right, then $\lim_{n\rightarrow \infty}f(x_n)=k+1$. Therefore, we have
$$\displaystyle{\liminf_{ n\rightarrow \infty}}f(x_n)\geq f(k).$$ 
This shows that $f$ is an $\mathcal{R}$-lower semi-continuous function but it is not $\mathcal{R}$-continuous.
\end{example} 

The next illustrative example shows that $\mathcal{R}$-LSC is in fact weaker than lower-semi-continuity.
\begin{example}
Let $X=[0,\infty)$ and  $d$ be the standard metric on $X$. We define $(x, y)\in \mathcal{R} ~\mbox{ if }~xy\leq x ~\mbox{or}~ y.$  Let $f:X\rightarrow X$ be defined as
 \[ f(x) = \left\{ \begin{array}{ll}
2 & {x\in [0,1)};\\
 1 & {x=1 };\\
 \frac{1}{2}&{x>1}.\end{array} \right. \]
We show that this function is neither lower semi-continuous nor $\mathcal{R}$-continuous but it is an $\mathcal{R}$-lower semi-continuous function. We consider the point $x=1$. Let $(x_n)$ be a non-constant sequence converging to $1$. So we have either $f(x_n)=2$ or $f(x_n)=\frac{1}{2}$ for all $n\in \mathbb{N}$ which shows that $$\displaystyle{\liminf_{ n\rightarrow \infty}}f(x_n)\geq f(1)$$ does not hold always. Hence, it is not a lower semi-continuous function at $x=1$. Similarly, one can check that this is not $\mathcal{R}$-continuous. Next, we show that this is an $\mathcal{R}$-lower semi-continuous function. Let us consider $(x_n)$ be an $\mathcal{R}$-preserving sequence converging to $1$. Then, for all $n\in \mathbb{N}, (x_n, x_{n+1})\in \mathcal{R}\Rightarrow x_nx_{n+1}\leq x_n~\mbox{or}~ x_{n+1}$ implies the following two cases:
 \begin{enumerate}
 \item $x_n=1$ for all $n\in \N$ and $f(x_n)=1=f(1)$.
 \item If $(x_n)$ be a non-constant $\mathcal{R}$-preserving sequence, then for all $n\in \N$, we must have $x_n<1$  and $f(x_n)=2$. Therefore,  $$\displaystyle{\liminf_{ n\rightarrow \infty}}f(x_n)\geq f(1).$$ 
\end{enumerate}
This implies that $f$ is an $\mathcal{R}$-lower semi-continuous function.
\end{example}
From the above two examples it is clear that $\mathcal{R}$-LSC is weaker than $\mathcal{R}$-continuity as well as lower-semi-continuity.
\begin{remark}
Every lower semi-continuous function is $\mathcal{R}$-lower-semi-continuous but the converse is not true. If $\mathcal{R}$ is  a universal relation, then the notions of lower-semi-continuity and $\mathcal{R}$-lower-semi-continuity will  coincide.
\end{remark} 
Now, we modify the definition of $w$-distance (Definition-\ref{w}) and the corresponding Lemma $1$ presented in \cite{kada} in the context of metric spaces endowed with an arbitrary binary relation $\mathcal{R}$.
\begin{definition} \label{w1}
Let $(X,d)$ be a metric space and $\mathcal{R}$ be a binary relation on $X$. A function $p:X\times X\rightarrow [0,\infty)$ is said to be a $w$-distance on $X$ if 
\begin{enumerate}
\item[$(w1')$] $p(x,z)\leq p(x,y)+p(y,z)$ for any $x,y,z\in X$,
\item[$(w2')$] for any $x\in X$, $p(x,.):X\rightarrow [0,\infty)$ is $\mathcal{R}$-lower semi-continuous,
\item[$(w3')$] for any $\epsilon>0$, there exists $\delta>0$ such that $p(z,x)\leq \delta$ and $p(z,y)\leq \delta$ imply $d(x,y)\leq \epsilon$.
\end{enumerate}
\end{definition} 
To prove our main results, we need the following lemma.
\begin{lemma}\label{L2}
Let $(X,d)$ be a metric space endowed with binary relation $\mathcal{R}$ and $p:X\times X\rightarrow [0,\infty)$ be a $w$-distance. Suppose $(x_n)$ and $(y_n)$ are two $\mathcal{R}$-preserving sequences in $X$ and $x,y,z\in X$. Let $(u_n)$ and $(v_n)$ be sequences of positive real numbers converging to $0$. Then we have the followings:
\begin{enumerate}
\item[(L1)]If $p(x_n,y)\leq u_n$ and $p(x_n,z)\leq v_n$ for all $n\in \N$, then $y=z$. Moreover, if $p(x,y)=0$ and $p(x,z)=0$, then $y=z$.
\item[(L2)] If $p(x_n,y_n)\leq u_n$ and $p(x_n,z)\leq v_n$ for all $n\in \N$, then $(y_n)\rightarrow z$.
\item[(L3)] If $p(x_n,x_m)\leq u_n$ for all $m>n$, then $(x_n)$ is an  $\mathcal{R}$-preserving Cauchy sequence in $X$.
\item[(L4)] If $p(x_n,y)\leq u_n$ for all $n\in \N$, then $(x_n)$ is an  $\mathcal{R}$-preserving Cauchy sequence in $X$.
\end{enumerate}
\end{lemma}
\begin{proof}
Proof is omitted as it can done be in the line of Lemma $1$ in \cite{kada}.
\end{proof}
\begin{remark}
Under the universal binary relation $\mathcal{R}$, Definition \ref{w1} will coincide with Definition \ref{w} and the Lemma \ref{L2} will coincide with Lemma 1 in \cite{kada}.
\end{remark}
Now we are in a position to state our main results. Before starting these, we highlight our main objectives which rest on the following considerations:
\begin{enumerate}
\item[$\bullet$] We refine the main result of Alam and Imdad (Theorem 3.1 in \cite{alm}) by considering more general distance function ($w$-distance) instead of the standard distance function on metric space endowed with an arbitrary binary relation and correspondingly we use a more general contraction principle. 
\item[$\bullet$] We present some non-trivial examples which lead to realize the sharpness of our obtained results.
\item[$\bullet$] Finally, we present an application to establish the existence and uniqueness of solutions of nonlinear fractional differential equations. 
\end{enumerate}
\baselineskip .55 cm

\section{\bf Main Results}
We start this section by extending the relation-theoretic version of Banach contraction principle owing to $w$-distance.
\begin{theorem}\label{thm1}
Let $(X,d)$ be a metric space with a $w$-distance $p$ and $\mathcal{R}$ be any arbitrary binary relation on $X$. Suppose $T$ is a self-map on $X$ with following conditions: 
\begin{enumerate}
\item there exists $Y\subseteq X~\mbox{ with }~T(X)\subseteq Y$ such that $(Y,d)$ is $\mathcal{R}$-complete,
\item $X(T,\mathcal{R})\neq \phi$ and $\mathcal{R}$ is $T$-closed,
%\item 
\item either $T$ is $\mathcal{R}$-continuous or $\mathcal{R}|_Y$ is $d$-self-closed,
\item there exists $ \lambda \in [0,1)$ such that 
$$p(Tx,Ty)\leq \lambda p(x,y)~~~\forall x,y\in X ~\mbox{with}~(x,y)\in \mathcal{R}$$
\end{enumerate}
then $F(T)\neq \phi.$    
\end{theorem}
\begin{proof}
As $X(T,\mathcal{R})\neq \phi$, so there exists a point $x_0\in X(T,\mathcal{R}) $ such that $(x_0,Tx_0)\in \mathcal{R}$. Now, we define a sequence $(x_n)$ by $x_n=T(x_{n-1})=T^n(x_0)$. By the property of $T$-closedness of $\mathcal{R}$, one can easily check that $(x_n)$ is an $\mathcal{R}$-preserving sequence that is  $$(x_n,x_{n+1})\in \mathcal{R}~\mbox{for all}~ n\in \mathbb{N}\cup\{0\}.$$ Applying the contraction principle of above theorem, we derive
\begin{eqnarray*}
p(Tx_{n-1},Tx_{n}) &\leq & \lambda p(x_{n-1},x_{n})\\
 \Rightarrow p(x_{n},x_{n+1})&\leq & \lambda p(x_{n-1},x_{n})\\
 &\leq & \lambda^2 p(x_{n-2},x_{n-1})\\
 &\vdots &\\
 &\leq & \lambda^n p(x_{0},x_1). 
\end{eqnarray*}
Using this for all $m>n$, we have,
\begin{eqnarray}
p(x_n,x_{m}) &\leq & p(x_n,x_{n+1})+p(x_{n+1},x_{n+2})+\cdots+p(x_{m-1},x_{m})\nonumber\\
&\leq & p(x_0,,x_{1})[\lambda^n+\lambda^{n+1}+\cdots +\lambda^{m-1}]\nonumber\\
&\leq & \frac{\lambda^n}{1-\lambda} p(x_0,x_{1}).\label{e1}
\end{eqnarray}
Let us define $u_n=\frac{\lambda^n}{1-\lambda} p(x_0,x_{1})$. Clearly $u_n\rightarrow 0$ as $n\rightarrow \infty$. So by $(L3)$, we must have that $(x_n)$ is an  $\mathcal{R}$-preserving Cauchy sequence in $Y$. Since $(Y,d)$ is $\mathcal{R}$-complete, so $x_n\rightarrow \tilde{x}$ as $n\rightarrow \infty$ for some $\tilde{x}\in Y$. 

Next, we show that $\tilde{x}$ is a fixed point of $T$. In order to prove this, at first we consider that $T$ is $\mathcal{R}$-continuous.

By using $\mathcal{R}$-continuity of $T$, we obtain
$$d(\tilde{x},T\tilde{x})=\displaystyle{\lim_{n\rightarrow \infty}}d(x_{n+1},T\tilde{x})=\displaystyle{\lim_{n\rightarrow \infty}}d(T(x_{n}),T\tilde{x})=d(T\tilde{x},T\tilde{x})=0.$$
This shows that $\tilde{x}$ is a fixed point of $T$.

Alternatively, we consider that $\mathcal{R}|_Y$ is $d$-self-closed. So, we must have a subsequence $(x_{n_k})$ of $(x_n)$ with $(x_{n_k},\tilde{x})\in \mathcal{R}$ for all $k\in \N\cup \{0\}$. Combining the Equation \ref{e1} with $\mathcal{R}$-lower-semi-continuity of $p$, we get
$$p(x_{n_{k}+1},\tilde{x})\leq \displaystyle{\liminf_{k\rightarrow \infty}}p(x_{n_{k}+1},x_{n_{k}+m}) \leq \displaystyle{\liminf_{k\rightarrow \infty}}\frac{\lambda^{n_{k}-1}}{1-\lambda} p(x_0,x_1)=0.$$
Since $\mathcal{R}$ is $T$-closed and $(x_{n_k}, \tilde{x})\in \mathcal{R}$, so $$p(Tx_{n_k},T\tilde{x})\leq \lambda p(x_{n_k},\tilde{x})\leq \lambda \displaystyle{\liminf_{k\rightarrow \infty}}p(x_{n_k},x_{n_{k}+m}) \leq \displaystyle{\liminf_{k\rightarrow \infty}}\frac{\lambda^{n_{k}+1}}{1-\lambda} p(x_0,x_1)=0.$$
By $(L1)$ of Lemma \ref{L2}, we must have $T\tilde{x}=\tilde{x}$, i.e., $\tilde{x}$ is a fixed point of $T$.
\end{proof}
The following theorem ensures the uniqueness of fixed point of $T$. We like  to provide an additional condition to the hypotheses of Theorem \ref{thm1} to ensure that the fixed point in Theorem 2.1 is in fact unique if any of the following conditions holds.
\begin{eqnarray}
\bf{\mbox{ \textbf{For every }} x,y \in T(X), \exists z\in T(X) \mbox{ \textbf{such that} }(z,x), (z,y)\in \mathcal{R}.} \label{e1}\\
\bf{\mathcal{R}|_{TX} \mbox{ \textbf{is complete.} }} \label{e2}
\end{eqnarray}

\begin{theorem}\label{thm2} In addition to the hypotheses of Theorem \ref{thm1}, suppose that any of the condition (\ref{e1}) or condition (\ref{e2}) holds. Then we obtain the uniqueness of fixed point of $T$.
\end{theorem}
\begin{proof}
We prove the theorem by considering following two possible cases.

\noindent{\textbf{Case I:}} Let in addition to the hypotheses of Theorem \ref{thm1}, condition (\ref{e1}) hold.  Then, for any two fixed points $\tilde{x},\tilde{y}$  of $T$, there must be an element $z\in T(X)$ such that $$(z,\tilde{x})\in \mathcal{R} ~\mbox{and}~ (z,\tilde{y})\in \mathcal{R}.$$ As $\mathcal{R}$ is $T$-closed, so for all $n\in \mathbb{N}\cup\{0\}$, 
$$(T^n(z),\tilde{x})\in \mathcal{R} ~\mbox{and}~ (T^n(z),\tilde{y})\in \mathcal{R}.$$
Using contractivity condition of $T$, we get
$$p(T^n(z),\tilde{x})=p(T^n(z),T^{n}\tilde{x})\leq \lambda^{n}p(z,\tilde{x})$$
and
$$p(T^n(z),\tilde{y})=p(T^n(z),T^n\tilde{y})\leq \lambda^{n}p(x_0,\tilde{y}).$$
Let us consider $u_n=\lambda^{n+1}p(z,\tilde{x})$ and $v_n=\lambda^{n+1}p(z,\tilde{y}).$ Clearly $(u_n)$ and $(v_n)$ are two sequences of real numbers converging to $0$. Hence by $(L1)$ of Lemma \ref{L2}, we obtain $\tilde{x}=\tilde{y}$, i.e., $T$ has a unique fixed point.

\noindent{\textbf{Case II:}}
Let in addition to the hypotheses of Theorem \ref{thm1}, condition (\ref{e2}) hold. Suppose $\tilde{x},\tilde{y}$ are two fixed points of $T$. Then we must have $(\tilde{x},\tilde{y})\in \mathcal{R}$ or $(\tilde{y},\tilde{x})\in \mathcal{R}$. For $(\tilde{x},\tilde{y})\in \mathcal{R}$, we obtain
 $$p(\tilde{x},\tilde{y})=p(T(\tilde{x}),T(\tilde{y}))\leq \lambda p(\tilde{x},\tilde{y})<p(\tilde{x},\tilde{y})$$ which leads to a contradiction. Hence, we must have $\tilde{x}=\tilde{y}$.
 
 In similar way, if $(\tilde{y},\tilde{x})\in \mathcal{R}$, we have $\tilde{x}=\tilde{y}$.
\end{proof}
In order to signify the motivations of our investigation, we present following examples.
\begin{example}\label{ex4}
Let $(X,d)$ be a metric space where $X=[1,3)$ and d is the standard metric define on $X$. We define a binary relation $\mathcal{R}=\{(x,y)\in X^2:x\geq y\}$. Let $T$ be a self-map on $X$ defined by 
\[ T(x) = \left\{ \begin{array}{ll}
\frac{x}{2}, & { x\in [1,2)};\\
 2, & {x\in [2,3)}.\end{array} \right. \]
 Now we check the hypotheses of Theorem 3.1 given in Alam and Imdad \cite{alm}.
 \begin{enumerate} 
 \item Let $Y=[1,2]$. Then it is clear that $T(X)\subseteq Y$ and $(Y,d)$ is $\mathcal{R}$-complete.
 \item For $x=1, T(x)=\frac{1}{2}$ such that $(x,Tx)\in \mathcal{R}$, i.e., $X(T,\mathcal{R})\neq \phi$.
 \item We show that $\mathcal{R}|_Y$ is $d$-self-closed. Let $(x_n)$ be an $\mathcal{R}$-preserving sequence converges to $x$. So for all $n\in \N, (x_n,x_{n+1})\in \mathcal{R}$, i.e., $x_n\geq x_{n+1}$ for all $n\in \N$ which implies that $(x_n)$ is a decreasing sequence converging to $x$. So, we must have that $(x_n,x)\in \mathcal{R}$ for all $n\in \N$. Hence, $\mathcal{R}|_Y$ is $d$-self-closed. 
 \item Now we show that we can't employ the contraction principle given in Theorem 3.1 in Alam and Imdad \cite{alm}.
 
 For example, we consider $x=2,y=1$. Then, clearly $(x,y)\in \mathcal{R}$ and $Tx=2, ~Ty=\frac{1}{2}$. Then
 $d(Tx,Ty)=d(2,\frac{1}{2})=\frac{3}{2}$ and $d(x,y)=1$. So we can't find any $k\in [0,1)$ such that $$d(Tx,Ty)\leq kd(x,y)$$ holds. But if we choose a $w$-distance function $p$ as $p(x,y)=|x|+|y|$, then for all $x,y\in X$, we have $$p(Tx,Ty)\leq \lambda p(x,y)$$ with $(x,y)\in \mathcal{R}$ and for some $\lambda\in [0,1)$.
 \end{enumerate}
Hence all the hypotheses of our theorem satisfy and note that $x=2$ is a fixed point of $T$ and it is the unique fixed point.
 
 \noindent{\textbf{Note:}}  It is worth mentioning that the results of Ahmadullah et el. \cite{ahm} are more generalized and improved version than that of Alam and Imdad \cite{alm} 
 but still in that example, we can't employ the main result (Theorem 2.1) of Ahmadullah et al. \cite{ahm}. For $x=2,y=1$, we obtain:
 $$M_T(2,1)=\max\{d(2,1),d(1,\frac{1}{2}),d(2,2),\frac{d(1,2)+d(2,\frac{1}{2})}{2}\}=\frac{5}{4}.$$
 In Theorem 2.1 given in \cite{ahm}, as $\phi$ is a function with $\phi(t)<t,t>0$, so we can't find any function $\phi$ with that property so that $$d(Tx,Ty)\leq \phi(M_T(x,y))$$ holds.  Hence, we can't employ the results of Ahmadullah et al. \cite{ahm} in that example.
\end{example}
Next, we furnish another important example.
\begin{example}\label{ex5}
Let us consider the metric space $(X,d),$ where $X=[0,2]$, $d$ is the standard metric on $X$ and $(x,y)\in \mathcal{R}$ if $xy\leq x~\mbox{or}~y$. We define a $w$-distance $p:X\times X\rightarrow X$ by $p(x,y)=y$. Let us define a function $T:X\rightarrow X$ by
\[ T(x) = \left\{ \begin{array}{ll}
\frac{x}{3}, & {0\leq x\leq \frac{2}{3}};\\
 1-x, & {\frac{2}{3}<x< 1 };\\
 \frac{3}{4}, & x=1;\\
 x-\frac{1}{2},&{x>1}.\end{array} \right. \]
Now if $(x,y)\in \mathcal{R}$, then $xy\leq x~or~ y$. Let us consider $xy\leq x$. So we have the following cases:
 
\textbf{Case 1}: Let $x=0$. Then for any $y\in [0,2], ~(x,y)\in \mathcal{R}$. So we get:
\begin{enumerate}
\item[(i)] for $0\leq y\leq \frac{2}{3},$ then $Tx=0$ and $Ty=\frac{y}{3}$. So, $p(Tx,Ty)=Ty=\frac{y}{3}$ and $p(Tx,Ty)=\frac{y}{3}\leq \frac{1}{3}p(x,y)$,
\item[(ii)] if $\frac{2}{3}<y<1,$ then $Ty\in(0,\frac{1}{3})$ and $p(Tx,Ty)=1-y<y=p(x,y)$. In particular, $p(Tx,Ty)\leq kp(x,y),$ where $k\in [\frac{1}{2},1)$,
\item[(iii)]let $y=1$. Then $p(T0,T1)=\frac{3}{4}\leq \frac{3}{4}p(0,1)$,
\item[(iv)] for $y>1$, we have $p(Tx,Ty)=y-\frac{1}{2}\leq ky=kp(x,y),$ where $k\in [\frac{3}{4},1)$.
\end{enumerate}
 
\textbf{Case 2}: For all $y\in [0,2]$ and $x=0$, we have $p(Ty,Tx)=0=kp(y,x)$ for all $k\in [0,1)$.
 
 \textbf{Case 3}: Let $x\neq 0$. Then $y\leq 1$. So, we have:
 \begin{enumerate}
 \item[(i)] for $0\leq y\leq \frac{2}{3}$, $p(Tx,Ty)\leq \frac{1}{3}p(x,y)$,
 \item[(ii)] for $\frac{2}{3}<y< 1$, $p(Tx,Ty)\leq kp(x,y)$, where $k\in [\frac{1}{2},1)$,
 \item[(iii)] for $y=1$, $p(Tx,T1)=\frac{3}{4}\leq \frac{3}{4}p(x,1)$ for all $x\in X$,
 \item[(iv)] for $y\leq 1$ and $x>1$, we have $p(Ty,Tx)\leq kp(y,x),$ where $k\in [\frac{3}{4},1)$.
 \end{enumerate}
 The above three cases show that $T$ satisfies the condition (5) of Theorem \ref{thm1}.  Next, we check the remaining hypotheses of our theorem.
\begin{enumerate}
\item Let us consider $Y=[0,\frac{3}{2}]$. Then we must have $TX\subseteq Y$ and $\mathcal{R}|_Y$ is $\mathcal{R}$-complete.
\item Clearly, $X(T,\mathcal{R})\neq \phi.$
\item $\mathcal{R}$ is $T$-closed.
\item Note that $T$ is not $\mathcal{R}$-continuous at $x=1$ and $x=\frac{2}{3}$. But $\mathcal{R}$ is $d$-self-closed.
\item For any $x,y\in Y$, one can always find $z\in Y$ such that $(z,x),(z,y)\in \mathcal{R}$.  
\end{enumerate}
 We have already check that $T$ satisfies contractivity condition. So, all the hypotheses of our theorem satisfy. Note that $x=0$ is a fixed point of $T$ and it is the unique fixed point of $T$. 
 \end{example}
\begin{remark}
\begin{enumerate}
\item 
It is notable that the binary relation $\mathcal{R}$ considered in our example is not reflexive, irreflexive and transitive. Here, $\mathcal{R}$ satisfies only symmetrical condition.
\item
 It is interesting to note that the mapping $T$ in above example neither satisfies the contractive condition of Theorem 3.1 in Alam and Imdad \cite{alm} nor the contractive condition of Theorem 2.1  in Ahmadullah et al. \cite{ahm}.

 For example, we consider $x=1$ and $y=\frac{3}{4}$. Clearly, $(x,y),(y,x)\in \mathcal{R}$. Therefore, 
\begin{eqnarray}
M_T(x,y)&=& \max \{d(x,y), d(x,Tx),d(y,Ty),\frac{d(x,Ty)+d(y,Tx)}{2}\}\nonumber\\
&=& \max\{d(1,\frac{3}{4}), d(1,\frac{3}{4}),d(\frac{3}{4},\frac{1}{4}),\frac{d(1,\frac{1}{4})+d(\frac{3}{4},\frac{3}{4})}{2}\}\nonumber\\
&=&\max \{\frac{1}{4},\frac{1}{2},\frac{3}{8}\}\nonumber\\
&=&\frac{1}{2},
\end{eqnarray}
and $d(Tx,Ty)=d(\frac{3}{4},\frac{1}{4})=\frac{1}{2}$.

Since, in Theorem (2.1) of Ahmadullah \cite{ahm}, $\phi$ is an increasing function with $\phi(t)<t,~\mbox{for} ~t>0$, so the mapping $T$ does not satisfy the contractive condition of this theorem and hence we can't exploit this theorem to obtain fixed point. Again since the Theorem 2.1 of Ahmadullah et al. \cite{ahm} is improved version over Theorem 3.1 of Alam and Imdad \cite{alm} and also Theorem 2.1 of Samet and Turinici \cite{samet1} (for symmetric binary relation), so we can't also employ these results to get fixed point of $T$ in that example. 
\end{enumerate}
\end{remark}

Analysing above two examples it is transparent that our findings unveil another direction of relation-theoretic metrical fixed point results where the main result given in Alam and Imdad \cite{alm} (Theorem 3.1) does not work (even the main result of Ahmadullah et al.\cite{ahm} (Theorem 2.1) does not work here).
\begin{remark}
If we set $p(x,y)=d(x,y),$ in Theorem \ref{thm1}, then we obtain the Theorem 3.1 of Alam and Imdad \cite{alm}. Hence our Theorem \ref{thm1} is an improved and generalized version of relation-theoretic metrical fixed point theorem due to Banach contraction given in Alam and Imdad \cite{alm}. 

\end{remark}
\section{Application}
In this section we employ our main result in nonlinear fractional differential equations. Here, we find a solution for the following nonlinear fractional differential equation (see \cite{bale}) given by:
\begin{equation*}
^CD^{\beta}x(t)=f(t,x(t))~~~~~~~(0<t<1,1<\beta\leq 2),
\end{equation*} with boundary conditions
\begin{equation*}
x(0)=0, ~~~x(1)=-\int_0^{\eta}x(s)ds~~~(0<\eta<1),
\end{equation*} where $^CD^{\beta}$ stands for the Caputo fractional derivative of order $\beta$ which is defined as
\begin{equation*}
^CD^{\beta}f(t)=\frac{1}{\Gamma(n-\beta)}\int_0^t(t-s)^{n-\beta-1}f^n(s)ds ~~~~(n-1<\beta<n;n=[\beta]+1),
\end{equation*}
and
$f:[0,1]\times \mathbb{R}\rightarrow \mathbb{R}^+$ is a continuous function. We consider $X=C([0,1],\mathbb{R})$, the set of all continuous functions from $[0,1]$ into $\mathbb{R}$ with supremum norm $||x||_\infty=\displaystyle{\sup_{t\in [0,1]}}|x(t)|.$   So, $(X,||.||_\infty)$ is a Banach space. 

The Riemann-Liouville fractional integral of order $\beta$ (for detail, see \cite{sud}) is given by
\begin{equation*}
I^{\beta}f(t)=\frac{1}{\Gamma(\beta)}\int_0^t(t-s)^{\beta-1}f(s)ds, ~~~\beta>0.
\end{equation*}
At first, we present an appropriate form of a nonlinear fractional differential equation and then investigate the existence of a solution of the given problem through fixed point theorem. So, we consider the following fractional differential equation:
\begin{equation}\label{equ1}
^CD^{\beta}x(t)=f(t,x(t))~~~~~~(0<t<1,1<\beta\leq 2),
\end{equation} with the integral boundary conditions
\begin{equation*}
x(0)=0,~~~~ x(1)=-\int_0^{k}x(s)ds~~~(0<k<1),
\end{equation*} where
\begin{enumerate}
\item $f:[0,1]\times \mathbb{R}\rightarrow \mathbb{R}^+$ is continuous function,
\item $x(t):[0,1]\rightarrow \mathbb{R}$ is continuous 
\end{enumerate}satisfying the following conditions:
$$|f(s,x)-f(s,y)\leq L|x-y|$$ for all $t\in [0,1]$ and $\forall x,y\in X$ such that $x(t)y(t)\geq  0$ and $L$ is a constant such that $L\lambda<1$ and $$\lambda=\frac{1}{\Gamma(\beta+1)}+\frac{2}{\Gamma(\beta+1)(2+k^2)}+\frac{2k^{1+\beta}}{\Gamma(\beta+1)(2+k^2)}.$$ 
Then the differential equation \ref{equ1} has unique solution.
\begin{proof}
We consider the following binary relation on $X$:
$$(x,y)\in \mathcal{R} ~\mbox{if}~ x(t)y(t)\geq 0, \forall t\in [0,1].$$ We consider $d(x,y)=\displaystyle{\sup_{t\in[0,1]}}||x(t)-y(t)||$ for all $x,y\in X.$ So, $(X,d)$ is an $\mathcal{R}$- complete metric space.

We define a mapping $T:X\rightarrow X$ by:
\begin{eqnarray*}
Tx(t)&=&\frac{1}{\Gamma(\beta)}\int_0^t(t-s)^{\beta-1}f(s,x(s))ds+\frac{2t}{(2+k^2)\times \Gamma(\beta)}\int_0^1(1-s)^{\beta-1}f(s,x(s))ds\\
&+& \frac{2t}{(2+k^2)\Gamma(\beta)}\int_0^{k}\int_0^s(s-m)^{\beta-1}f(m,x(m))dmds 
\end{eqnarray*} for $t\in [0,1]$. 

A function $x\in X$ is a solution of Equation \ref{equ1} iff $x(t)=Tx(t)$ for all $t\in [0,1]$. In order to prove  the existence of fixed point of $T$, we  show that $\mathcal{R}$ is $T$-closed and $T$ satisfies the contractive condition.

At first, we show that $\mathcal{R}$ is $T$-closed. Let, for all $t\in [0,1], (x(t),y(t))\in \mathcal{R}$. Now, we have:
\begin{eqnarray*}
 Tx(t)&=&\frac{1}{\Gamma(\beta)}\int_0^t(t-s)^{\beta-1}f(s,x(s))ds+\frac{2t}{(2+k^2)\times \Gamma(\beta)}\int_0^1(1-s)^{\beta-1}f(s,x(s))ds\\
&+& \frac{2t}{(2+k^2)\Gamma(\beta)}\int_0^k(\int_0^s(s-m)^{\beta-1}f(m,x(m))dm)ds>0
\end{eqnarray*}
which implies that $(Tx,Ty)\in \mathcal{R}$, i.e., $\mathcal{R}$ is $T$-closed. Also, it is clear that for any $x(t)\geq 0, t\in [0,1]$, we have $Tx(t)\geq 0$ for all $t\in [0,1]$, i.e., $(x(t),Tx(t))\in \mathcal{R}$ for all $t\in [0,1]$ which implies that $X(T,\mathcal{R})\neq \phi$. 

Next, we show that $T$ satisfies the contraction condition. For all $t\in [0,1]$ and $(x(t),y(t))\in \mathcal{R}$, we obtain:
\begin{eqnarray*}
|Tx-Ty|&=&\big|\frac{1}{\Gamma(\beta)}\int_0^t(t-s)^{\beta-1}f(s,x(s))ds+\frac{2t}{(2+k^2)\Gamma(\beta)}\int_0^1(1-s)^{\beta-1}f(s,x(s))ds\\
&+& \frac{2t}{(2+k^2)\Gamma(\beta)}\int_0^k(\int_0^s(s-m)^{\beta-1}f(m,x(m))dm)ds\\
&-& \frac{1}{\Gamma(\beta)}\int_0^t(t-s)^{\beta-1}f(s,y(s))ds-\frac{2t}{(2+k^2)\Gamma(\beta)}\int_0^1(1-s)^{\beta-1}f(s,y(s))ds\\
&-& \frac{2t}{(2+k^2)\Gamma(\beta)}\int_0^k(\int_0^s(s-m)^{\beta-1}f(m,y(m))dm)ds\big|\\
& \leq & \frac{1}{\Gamma(\beta)}\int_0^t(t-s)^{\beta-1}\big|f(s,x(s))-f(s,y(s))\big|ds\\
&+&\frac{2}{(2+k^2)\Gamma(\beta)}\int_0^1(1-s)^{\beta-1}\big|f(s,x(s))-f(s,y(s))\big|ds\\
&+& \frac{2}{(2+k^2)\Gamma(\beta)}\int_0^k\int_0^s(s-m)^{\beta-1}\big|f(m,x(m))-f(m,y(m))\big|dmds\\
&\leq &\frac{L||x-y||}{\Gamma(\beta)}\int_0^t(t-s)^{\beta-1}ds+\frac{2L||x-y||}{(2+k^2)\Gamma(\beta)}\int_0^1(1-s)^{\beta-1}ds\\
&+& \frac{2L||x-y||}{(2+k^2)\Gamma(\beta)}\int_0^k\int_0^s(s-m)^{\beta-1}dmds\\
&\leq &\frac{L||x-y||}{\Gamma(\beta+1)}+\frac{2L||x-y||}{(2+k^2)\Gamma(\beta+1)}+ \frac{2k^{\beta+1}L||x-y||\Gamma(\beta)}{(2+k^2)\Gamma(\beta+2)}\\
&\leq & L||x-y||(\frac{1}{\Gamma(\beta+1)}+\frac{2}{(2+k^2)\Gamma(\beta+1)}+ \frac{2k^{\beta+1}\Gamma(\beta)}{(2+k^2)\Gamma(\beta+2)})\\
\Rightarrow ||Tx-Ty||& \leq & L\lambda ||x-y||.
\end{eqnarray*}
Now, if we set $p(x,y)=d(x,y)$, then we have
$$p(Tx,Ty)\leq L\lambda p(x,y)$$ which shows that $T$ satisfies the contraction condition as $L\lambda<1$.

Next, we consider that $(x_n)$ is an $\mathcal{R}$-preserving Cauchy sequence converging to $x$. So, we must have $x_n(t)x_{n+1}(t)\geq 0$ for all $t\in[0,1]$ and $n\in \mathbb{N}$. This gives us two possibilities: either $x_n(t)\geq 0$ or $x_n(t)\leq 0$ for all $n\in \mathbb{N}$ and each $t\in [0,1]$. Let us consider the case $x_n(t)\geq 0$ for each $t\in [0,1]$ and $n\in \mathbb{N}$. Then, for every $t\in [0,1]$, $x_n(t)$ produces a sequence of non-negetive real numbers which converges to $x(t)$. Hence, we must get $x(t)\geq 0$ for each $t\in [0,1]$, i.e., $(x_n(t),x(t))\in \mathcal{R}$ for all $n$ and $t\in [0,1]$. This shows that $\mathcal{R}$ is $d$-self-closed. So, by Theorem \ref{thm1}, $x(t)$ is a fixed point of $T$  which is the required solution of Equation \ref{equ1}.

Finally, we show that $x(t)$ is the unique solution of Equation \ref{equ1}. If possible, let $y(t)$ be another solution of Equation \ref{equ1} which implies that $Ty(t)=y(t)$ for all $t\in [0,1]$. Now, we consider a constant function $z(t)=0$ for all $t\in [0,1]$. Then, it is trivial to show that $(z(t),x(t))\in \mathcal{R}$ and $(z(t),y(t))\in \mathcal{R}$ for all $t\in [0,1]$. Hence, by Theorem \ref{thm2}, we claim that $x(t)$ is the unique solution of Equation \ref{equ1}.
\end{proof}

\vskip.5cm\noindent{\bf Acknowledgements}\\
 The first named author would like to express her sincere thanks to DST-INSPIRE, New Delhi, India for their financial supports under INSPIRE fellowship scheme.


\begin{thebibliography}{10}

\bibitem{ahm}
M.~Ahmadullah, M.~Imdad, and R.~Gubran.
\newblock Relation-theoretic metrical fixed point theorems under nonlinear
  contractions.
\newblock To appear in Fixed Point Theory.

\bibitem{alm}
A.~Alam and M.~Imdad.
\newblock Relation-theoretic contractive principle.
\newblock {\em J. Fixed Point Theory Appl.}, 17(4):693--702, 2015.

\bibitem{bale}
D.~Baleanu, S.~Rezapour, and H.~Mahammadi.
\newblock Some existence results on nonlinear fractional differential
  equations.
\newblock {\em Philos Trans R Soc A, Math Phys Eng Sci.}, 371, 2013.

\bibitem{h}
H.~Ben-El-Mechaiekh.
\newblock The {R}an-{R}eurings fixed point theorems without partial order: {A}
  simple proof.
\newblock {\em J. Fixed Point Theory Appl.}, 16:373--383, 2015.

\bibitem{ghod}
S.~Ghods, M.E. Gordji, M.~Ghods, and M.Hadian.
\newblock Comment on {``{C}oupled fixed point theorems for nonlinear
  contractions in partially ordered metric spaces"}[{L}akshmikantham and
  \'{C}iri\'c, nonlinear anal. tma 70 (2009) 4341-4349].
\newblock {\em J. Comput. Anal.}, 14(5):958--966, 2012.

\bibitem{kada}
O.~Kada, T.~Suzuki, and W.~Takahashi.
\newblock Nonconvex minimization theorems and fixed point theorems in complete
  metric spaces.
\newblock {\em Math. Japonica.}, 44(2):381--391, 1996.

\bibitem{mad}
B.~Kolman, R.~C. Busby, and S.~Ross.
\newblock {\em Relation Algebras, Studies in Logic and Foundations of
  Mathematics}.
\newblock 150, {E}lsevier B.V., Amsterdam, 2006.

\bibitem{lip}
S.~Lipschutz.
\newblock {\em Schaum's {O}utlines of {T}heory and {P}roblems of {S}et {T}heory
  and {R}elated {T}opics}.
\newblock McGraw-Hill, New York, 1964.

\bibitem{nie}
J.J. Nieto and R.~Rodr\'iguez-L\'opez.
\newblock Contractive mapping theorems in partially ordered sets and
  applications to ordinary differential equations.
\newblock {\em Order}, 22(3):1435--1443, 2005.

\bibitem{samet1}
B.~Samet and M.Turinici.
\newblock Fixed point theorems on a metric space endowed with an arbitrary
  binary relation and applications.
\newblock {\em Commun. Math. Anal.}, 13:82--97, 2012.

\bibitem{sud}
W.~Sudsutad and J.~Tariboon.
\newblock Boundary value problems for fractional differential equations with
  three-point fractional integral boundary conditions.
\newblock {\em Adv. Difference. Equ.}, 2012(93), 2012.

\bibitem{tur4}
M.~Turinici.
\newblock Ran-{R}eurings fixed point results in ordered metric spaces.
\newblock {\em Libertas Math.}, 31:49--55, 2011.

\bibitem{tur5}
M.~Turinici.
\newblock Nieto-{L}\'opez theorems in ordered metric spaces.
\newblock {\em Math. Student}, 81:219--229, 2012.

\bibitem{tur3}
M.~Turinici.
\newblock Linear contractions in product ordered metric spaces.
\newblock {\em Ann. Univ. Ferrara.}, 59(1):187--198, 2013.

\end{thebibliography}
\end{document}